\documentclass[12pt]{article}
\usepackage{mathrsfs}
\usepackage{graphicx,amsmath,amsfonts,amssymb,amscd,amsthm}
\newtheoremstyle{kai}
{3pt} {3pt} {} {} {\bfseries} {.} {.5em} {}
\def\EquationsBySection{\def\theequation
{\thesection.\arabic{equation}}%
\@addtoreset{equation}{section}}

\newcommand\old[1]{}
\newcommand{\pend}{\hfill \thicklines \framebox(6.6,6.6)[l]{}}
\newcommand{\E}{\mathbb{E}}
\renewenvironment{proof}{\noindent {\it  Proof.} \rm}{\pend}
\newtheorem{theorem}{Theorem}[section]
\newtheorem{lemma}{Lemma}[section]

\newtheorem{remark}{Remark}[section]
\newtheorem{definition}{Definition}[section]
\newtheorem{example}{Example}[section]

\EquationsBySection \makeatother
\usepackage{indentfirst}

\begin{document}
\pagestyle{plain}
\title
{\bf Lyapunov Exponents of Hybrid Stochastic Heat Equations}

\author{Jianhai Bao$^{a,}$\thanks{{\it E-mail address:}
jianhaibao@yahoo.com.cn, xuerong@stams.strath.ac.uk,
C.Yuan@swansea.ac.uk},  Xuerong Mao$^b$, Chenggui
Yuan$^{a}$\\
\\
${}^a$Department of Mathematics,\\
 Swansea University, Swansea SA2 8PP, UK \vspace{2mm}\\
 ${}^b$Department of Mathematics and Statistics,\\
University of Strathclyde, Glasgow  G1 1XH, UK\vspace{2mm}\\
}

\date{}
\maketitle
\begin{abstract}{\rm In this paper, we investigate a class of hybrid
stochastic heat equations.  By explicit formulae of solutions, we
not only reveal the sample Lyapunov exponents but also discuss  the
$p$th moment Lyapnov exponents. Moreover, several examples are
established to demonstrate that unstable (deterministic or
stochastic) dynamical systems can be stabilized by Markovian
switching.
 }\\

\end{abstract}
\noindent

\section{Introduction}
Stabilization of (ordinary) stochastic differential equations (SDEs)
by noise has been studied extensively in the past few years, e.g.,
Arnold et al. \cite{a83}, Has'minskii \cite{k80}, Mao and Yuan
\cite{my06}, Pardoux and Wihstutz \cite{pw88,pw92}, Scheutzow
\cite{s93}. Recently, there are also many works focusing on such
phenomena for stochastic partial differential equations (SPDEs),
e.g., Kwiecinska \cite{k99} and Kwiecinska \cite{k02} discussed by
multiplicative noise
 such problems for heat
equations and a class of deterministic evolution equations,
respectively; some results on almost sure exponential stabilization
of SPDEs were established in Caraballo et al. \cite{clm01} by a
Lyapunov function argument;  stabilization by additive noise on
solutions to semilinear parabolic
 SPDEs with quadratic nonlinearities
 was investigated due to Bl\"omker et al. \cite{bhp07}.

Also, there has been increasing attention to hybrid SDEs (also known
as SDEs with Markovian switching and please see, e.g., the
monographs \cite{my06,yz10}), in which continuous dynamics are
intertwined with discrete events. One of the distinct features of
such systems is that the underlying dynamics are subject to changes
with respect to certain configurations while the continuous-time
Markov chains are used to delineate many practical systems, where
they may experience abrupt changes in their structure and
parameters.

It should be pointed out that there are some papers discussing the
almost sure exponential stability and the $p$th moment exponential
stability of stochastic heat equations without Markov switching,
e.g., Kwiecinska \cite{k99,k01,k02} and Xie \cite{x08}. However,
hybrid SPDEs,  e.g., hybrid stochastic heat equations, have so far
little been studied. In this paper, we are mainly interested in
finding explicit formulae of solutions for a class of hybrid
stochastic heat equations, and compute explicitly the sample
Lyapunov exponents as well as the $p$th moment Lyapunov exponents.
In particular,  due to the involvement of Markov chains, it is much
more complicated to study the $p$th moment Lyapunov exponents of
hybrid stochastic heat equations, and our results can not get
straightforward from \cite{k01, x08}. To cope with the difficulties
arising from the Markov switching, large deviation technique has
been used in this paper. For large deviation of Markov processes, we
refer to, e.g., Donsker and Varadhan \cite{dv75}, Wu \cite{wu00} and
the references therein. On the other hand, we also use our theories
to reveal the stabilization of (stochastic) dynamical systems by
Markovian switching.

The organization of this paper will be arranged as follows: In
Section 2 we recall some notation and notions, and give an
existence-and-uniqueness result for hybrid stochastic heat
equations. We investigate, in Section 3, the explicit formulae of
solutions for hybrid stochastic heat equations, and, in Section 4,
the sample Lyapunov exponents, where two examples are constructed to
show that Markovian switching can be used to stabilize unstable
stochastic dynamical systems. The last section reveals by the large
deviation principle the $p$th moment Lyapunov exponents.

\section{Preliminaries}

Let $\{\Omega,{\mathcal F},\{{\mathcal F}_{t}\}_{t\geq0},
\mathbb{P}\}$ be a complete probability space with a filtration
satisfying the usual conditions, and $B(t)$, $t\ge 0$,  a
real-valued Brownian motion defined on the above probability space.
For a bounded domain $\mathcal {O}\subset\mathbb{R}^n$ with
$C^{\infty}$ boundary $\partial\mathcal {O}$, let $L^2(\mathcal
{O})$ denote the family of all real-valued square integrable
functions, equipped with the usual inner product $\langle
f,g\rangle:=\int_{\mathcal {O}}f(x)g(x)dx, f,g\in L^2(\mathcal
{O})$, and the norm $\|f\|:=\left(\int_{\mathcal
{O}}f^2(x)dx\right)^{\frac{1}{2}}, f\in L^2(\mathcal {O})$. Let
$A:=\sum_{i=1}^n\frac{\partial^2}{\partial x^2}$ be the Laplace
operator, with domain $\mathcal {D}(A):=H^1_0(\mathcal {O})\cap
H^2(\mathcal {O})$,  where $H^m(\mathcal {O}), m=1,2$, consist of
functions of $L^2(\mathcal {O})$ whose derivatives $D^{\alpha}u$  of
order $|\alpha|\leq m$ are in $L^2(\mathcal {O})$, and
$H^m_0(\mathcal {O})$ is the subspace of elements of $H^m(\mathcal
{O})$ vanishing  on $\partial\mathcal {O}$. Furthermore, let
$\{e_n\}_{n\geq1}$, forming an orthonormal basis of $L^2(\mathcal
{O})$, and $\{\lambda_n\}_{n\geq1}\uparrow\infty$ be the eigenvector
and eigenvaue of $-A$ respectively, namely
\begin{equation}\label{eq00}
 -Ae_n=\lambda_ne_n.
\end{equation}
Thus, for any $f\in L^2(\mathcal {O})$, we can write $
f=\sum_{n=1}^{\infty}f_ne_n, \mbox{ where } f_n:=\langle f,
e_n\rangle. $

Let $\mathbb{R}_{+}:=[0,\infty)$ and $N$ be some positive integer.
Let $\{r(t),t\in \mathbb{R}_{+}\}$
  be a right continuous Markov chain on
the probability space $\{\Omega,{\mathcal F},\{{\mathcal
F}_{t}\}_{t\geq0},  \mathbb{P}\}$ taking values in a finite state
space $\mathbb{S}:=\{1,2,...,N\}$, with generator
$\Gamma:=(\gamma_{ij})_{N\times N}$ given by
$$ \mathbb{P}(r(t + \Delta )=j|r(t)=i) = \left\{
\begin{array}{cc}
 \gamma_{ij} \Delta  + o(\Delta ),\ \ & \ \ {\rm  if}\ \ i \ne j, \\
 1 + \gamma_{ii} \Delta  + o(\Delta ),\ \ & \ \ {\rm  if}\ \ i = j, \\
 \end{array} \right.$$
where $\Delta>0$ and $\gamma_{ij}\geq 0$ is
 the transition rate from $i$ to $j$, if $i\ne j$; while $\gamma_{ii}=-\sum_{j\ne
 i}\gamma_{ij}$.  We assume that Markov chain $r(\cdot)$ is independent
 of  Brownian motion $B(\cdot)$. It is known that almost every sample path of $r(t)$
 is a right continuous step function with a finite number of sample
 jumps in any finite sub-interval of $\mathbb{R}_{+}$. We further assume
 that the Markov chain $r(t)$ is irreducible. This is equivalent to
 the condition that, for any $i,j\in\mathbb{S}$, one can find finite
 numbers $i_1,i_2,\cdots,i_k\in\mathbb{S}$ such that
$ \gamma_{i,i_1}\gamma_{i_1,i_2}\cdots\gamma_{i_k,j}>0. $ The
algebraic interpretation of irreducibility is
\mbox{ran}$(\Gamma)=N-1$. Under this condition, the Markov chain
$r(t)$ has a unique stationary probability distribution
$\pi:=(\pi_1,\pi_2,\cdots,\pi_N)$ which can be determined by solving
the following linear equation $ \pi\Gamma=0 $ subject to $
\sum_{j=1}^{N}\pi_j=1  \mbox{ and } \pi_j>0, \forall j\in\mathbb{S}.
$

In this paper we consider hybrid stochastic heat equation:
\begin{equation}\label{eq1}
\begin{cases}
\frac{\partial u(t,x)}{\partial t}=A
u(t,x)+\alpha(r(t))u(t,x)+\beta(r(t))u(t,x)\dot{B}(t),
  &x\in\mathcal {O}, \ t > 0;\\
u(t,x)=0,  &x\in\partial\mathcal {O},\ t>0;\\
u(0,x)=u^0(x),  &x\in\mathcal {O}.
\end{cases}
\end{equation}
Here $u^0$ is a
 $\mathcal {D}(A)$-valued $\mathcal {F}_0$-measurable random
 variable such that $\mathbb{E}\|u^0\|^p<\infty$ for any $p>0$,
and  independent of $r(\cdot)$ and $B(\cdot)$.   Moreover,
 $\alpha,\beta$ are mappings from
$\mathbb{S}\rightarrow\mathbb{R}$ and we shall write
$\alpha_i:=\alpha(i),\ \beta_i:=\beta(i)$ in what follows.  In
general, $u(t,x)$ is referred to as the state and $r(t)$ is regarded
as the mode.  In its operation, hybrid system will switch from one
mode to another according to the law of Markov chain. In this paper,
we shall write $u(t):=u(t,\cdot)$, and $u:=\{u(t)\}_{t\in
\mathbb{R}_+}$. Hence, for each $t\geq 0$, $u(t)$ is an
$L^2(\mathcal {O})$-valued random variable while $u$ is an
$L^2(\mathcal {O})$-valued stochastic process.  Let us fix an
interval $[0,T]$ for arbitrary $T>0$ and recall the definition of
solution to Eq. \eqref{eq1} from Da Prato et al. \cite{pit82}.

\begin{definition}
{\rm  An $L^2(\mathcal {O})$-valued predictable process
$u=\{u(t)\}_{t\in [0,T]}$ is called the solution to Eq. \eqref{eq1}
if the following conditions are satisfied:
\begin{enumerate}
\item[\textmd{(i)}]
$u\in C([0,T];L^2(\mathcal {O}))$ and, for any $t\in[0,T]$,
$u(t)\in\mathcal {D}(A)$ a.s.;
\item[\textmd{(ii)}]
Stochastic integral equation
\begin{equation*}
\begin{split}
u(t,x)=u^0(x)+\int_0^t[A
u(s,x)+\alpha(r(s))u(s,x)]ds+\int_0^t\beta(r(s))u(s,x)dB(s)
\end{split}
\end{equation*}
 holds a.s. for any
$t\in[0,T]$ and $x\in\mathcal {O}$.
\end{enumerate}
}
\end{definition}

\begin{theorem}\label{theorem1}
{\rm   Eq. \eqref{eq1} admits a unique  solution $u=\{u(t)\}_{t\in
\mathbb{R}_+}$ such that $ \E\|u(t)\|^p < \infty, t >0, p>0. $}
\end{theorem}

\begin{proof}
Recall that almost every sample path of $r(\cdot)$ is a
right-continuous step function with a finite number of sample jumps
in any finite time. So there is a sequence $\{\tau_k\}_{k\geq 0}$ of
stopping times such that $\tau_0=0$,  $\lim_{k\to
\infty}\tau_k=\infty$ a.s. and $ r(t)=r(\tau_k) \mbox{ on }
\tau_k\leq t<\tau_{k+1} \mbox{ for } \forall k\geq0. $ Let $T>0$ be
arbitrary . We first consider Eq. \eqref{eq1} on
$t\in[0,\tau_1\wedge T)$ which becomes
\begin{equation*}
\begin{cases}
\frac{\partial u(t,x)}{\partial t}=A
u(t,x)+\alpha(r(0))u(t,x)+\beta(r(0))u(t,x)\dot{B}(t),  &x\in\mathcal {O};\\
u(t,x)=0, &x\in\partial\mathcal {O};\\
u(0)=u^0(x), &x\in\mathcal {O}.
\end{cases}
\end{equation*}
By \cite[Proposition 1]{k99}, Eq. \eqref{eq1} has a unique solution
$u\in C([0,\tau_1\wedge T];L^2(\mathcal {O}))$ and $u(t)\in\mathcal
{D}(A)$ for $t\in[0,\tau_1\wedge T]$ a.s. while $\E\|u(\tau_1\wedge
T)\|^p <\infty$. Setting $u^1(x):=u(\tau_1\wedge T,x)$, we next
consider Eq. \eqref{eq1} on $t\in[\tau_1\wedge T,\tau_2\wedge T)$
which becomes
\begin{equation*}
\begin{cases}
\frac{\partial u(t,x)}{\partial t}=A
u(t,x)+\alpha(r(\tau_1))u(t,x)+\beta(r(\tau_1))u(t,x)\dot{B}(t), &x\in\mathcal {O};\\
u(t,x)=0, &x\in\partial\mathcal {O};\\
u(\tau_1\wedge T,x)=u^1(x), &x\in\mathcal {O}.
\end{cases}
\end{equation*}
Again by \cite[Proposition 1.]{k99},  Eq. \eqref{eq1} has a unique
solution $u\in C([\tau_1\wedge T,\tau_2\wedge T];L^2(\mathcal {O}))$
and $u(t)\in\mathcal {D}(A)$ for $t\in[\tau_1\wedge T,\tau_2\wedge
T]$  a.s. while $\E\|u(\tau_2\wedge T)\|^p <\infty$. Repeating this
procedure, we see that Eq. \eqref{eq1} has a unique  solution $u \in
C([0,T];L^2(\mathcal {O}))$, and, for any $t\in [0,T]$,
$u(t)\in\mathcal {D}(A)$ a.s. while $\E\|u(T)\|^p <\infty$. Since
$T$ is arbitrary, the proof is complete.
\end{proof}

We conclude this section by defining the sample Lyapunov exponent
and the $p$th moment Lyapunov exponent for the solution of Eq.
\eqref{eq1}

\begin{definition}
 {\rm The limit
\begin{equation*}
\lambda(u^0):=\limsup_{t\rightarrow\infty}\frac{1}{t}\log(\|u(t)\|)
\end{equation*}
is called the sample Lyapunov exponent, while
 the limit
\begin{equation*}
\gamma_p(u^0):=\limsup\limits_{t\rightarrow\infty}\frac{1}{t}\log
\mathbb{E}(\|u(t)\|^p)
\end{equation*}
 is called the $p$th moment Lyapunov exponent.}
\end{definition}

If $\lambda(u^0) <0$ a.s. for any initial data $u^0$ (obeying the
conditons imposed above of course), then any solution of Eq.
(\ref{eq1}) will converge to zero exponentially with probability
one.  In this case, we say that the solution of Eq. (\ref{eq1}) is
almost surely exponentially stable.  Similarly, if $\gamma_p(u^0)
<0$ for any $u^0$, then the  solution of Eq. (\ref{eq1}) is
exponentially stable in the $p$th moment.

\section{Explicit Solutions of Hybrid  Stochastic Heat Equations}

Let us first discuss  hybrid stochastic heat equation with external
forces and Dirichlet boundary conditions
\begin{equation}\label{eq2}
\begin{cases}
\frac{\partial v(t,x)}{\partial t}=A v(t,x)+\alpha(r(t))v(t,x),\
& x\in\mathcal {O},\ t >0;\\
v(t,x)=0, &x\in\partial\mathcal {O}, \ t>0; \\
v(0,x)=u^0(x), & x\in\mathcal {O}.
\end{cases}
\end{equation}
 In the sequel, we shall denote $u_n^0:=\langle
u^0,e_n\rangle$ for $n\geq1$.  In general, $u_n^0$ is a random
variable but it becomes a (non-random) number if $u^0$ is
deterministic. When $u^0$ is deterministic and $u^0\not= 0$, we set
$n_0:=\inf\{n:u^0_n\neq 0\}$.

\begin{theorem} \label{thm3.1}
 {\rm The unique  solution of Eq. \eqref{eq2} has the explicit form
\begin{equation}\label{eq6}
v(t,x)=\sum\limits_{n=1}^{\infty}\exp\Big\{-\lambda_nt+\int_0^t\alpha(r(s))ds\Big\}u_n^0
e_n(x), \ \ \ t\geq0, x\in\mathcal {O}.
\end{equation}
}
\end{theorem}

\begin{proof}
Clearly, by Theorem \ref{theorem1} we can conclude that \eqref{eq2}
has a unique solution $v=\{v(t)\}_{t\geq0}$. Let $\{\tau_k\}_{k\ge
0}$ be the same as defined in the proof of Theorem \ref{theorem1}.
For $t\in[0,\tau_1)$ Eq. \eqref{eq2} can be written as
\begin{equation*}
\begin{cases}
\frac{\partial v(t,x)}{\partial t}=A v(t,x)+\alpha(r(0))v(t,x),\
& x\in\mathcal {O};\\
v(t,x)=0,  &x\in\partial\mathcal {O};\\
v(0,x)=u^0(x), & x\in\mathcal {O}.
\end{cases}
\end{equation*}
It is well known (see, e.g., \cite{k99, pz92}) that this heat
equation has  explicit solution on $t\in[0,\tau_1]$
\begin{equation}\label{eq3}
v(t,x)=\sum\limits_{n=1}^{\infty}\exp\{(-\lambda_n+\alpha(r(0))t)\}u_n^0
e_n(x).
\end{equation}
Set $v^1(x):=v(\tau_1,x)$ and consider
 Eq. \eqref{eq2} for $t\in[\tau_1,\tau_2)$, namely
\begin{equation*}
\begin{cases}
\frac{\partial v(t,x)}{\partial t}=A
v(t,x)+\alpha(r(\tau_1))v(t,x),\
 &x\in\mathcal {O};\\
v(t,x)=0, & x\in\partial\mathcal {O};\\
v(\tau_1,x) = v^1(x), &x\in\mathcal {O}.
\end{cases}
\end{equation*}
Again this heat equation has explicit solution on
$t\in[\tau_1,\tau_2)$
\begin{equation}\label{eq4}
v(t,x)=\sum\limits_{n=1}^{\infty}\exp\{(-\lambda_n+\alpha(r(\tau_1))(t-\tau_1))\}\langle
v^1,e_n\rangle e_n(x).
\end{equation}
Letting $t=\tau_1$ in \eqref{eq3} and substituting $v^1$ into
\eqref{eq4}, we derive that
\begin{equation*}
\begin{split}
v(t,x)&=\sum\limits_{n=1}^{\infty}\exp\{(-\lambda_n+\alpha(r(\tau_1))(t-\tau_1))\}\\
&\times\Big\langle
\sum\limits_{j=1}^{\infty}\exp\{(-\lambda_j+\alpha(r(0))\tau_1)\}u_j^0 e_j,e_n\Big\rangle e_n(x)\\
&=\sum\limits_{n=1}^{\infty}\exp\{(-\lambda_n+\alpha(r(\tau_1))(t-\tau_1))\}\exp\{(-\lambda_n+\alpha(r(0))\tau_1)\}u^0_n e_n(x)\\
&=\sum\limits_{n=1}^{\infty}\exp\{-\lambda_nt+\alpha(r(0))\tau_1+\alpha(r(\tau_1))(t-\tau_1)\}u^0_n e_n(x)\\
&=\sum\limits_{n=1}^{\infty}\exp\Big\{-\lambda_nt+\int_0^t\alpha(r(s))ds\Big\}u^0_n
e_n(x).
\end{split}
\end{equation*}
Repeating this  procedure, we  obtain the required assertion
(\ref{eq6}). The proof is complete.
\end{proof}

\begin{theorem}\label{theorem2}
{\rm  The solution of Eq. \eqref{eq2} has the following properties:
\begin{enumerate}
\item[\textmd{(1)}]
 If $u^0$ is deterministic and $u^0\not= 0$, then
\begin{equation}\label{eq5}
\lim\limits_{t\rightarrow\infty}\frac{1}{t}\log(\|v(t)\|)
=-\Big(\lambda_{n_0}-\sum\limits_{j=1}^{N}\pi_j\alpha_j\Big) \quad
\mbox{ a.s.}
\end{equation}
 In particular,  the solution  of Eq.
\eqref{eq2} with initial data $u^0$ will converge exponentially to
zero with probability one if and only if
\begin{equation*}
\lambda_{n_0}-\sum\limits_{j=1}^{N}\pi_j\alpha_j>0.
\end{equation*}
\item[\textmd{(2)}]For any initial data  $u^0$,
\begin{equation}\label{eq13}
\limsup\limits_{t\rightarrow\infty}\frac{1}{t}\log(\|v(t)\|)
 \leq-\Big(\lambda_1-\sum\limits_{j=1}^{N}\pi_j\alpha_j\Big)
 \quad \mbox{ a.s.}
  \end{equation}
In particular, the  solution  of Eq. \eqref{eq2} is almost surely
exponentially  stable if
\begin{equation*}
\lambda_1-\sum\limits_{j=1}^{N}\pi_j\alpha_j>0.
\end{equation*}
\end{enumerate}
}
\end{theorem}

\begin{remark}
{\rm In case (1) above, since $\lambda_{n_0}$ depends on the initial
data $u^0,$
  we only know the asymptotic behavior of  the solution with initial data $u^0$.
   However, in case (2), the estimate on the sample Lyapunov exponent holds
   for any initial data whence
    the   solution  of Eq. \eqref{eq2} is almost surely exponentially  stable
    if the right-hand-side term of \eqref{eq13} is negative.}
\end{remark}

\begin{proof} (1)
For any $t>0$,  it follows from \eqref{eq6} that
\begin{eqnarray}\label{eq7}
\frac{1}{t}\log(\|v(t)\|) &=
&\frac{1}{t}\log\Big(\sum\limits_{n=n_0}^{\infty}
\Big|\exp\Big\{-\lambda_{n_0}t+\int_0^t\alpha(r(s))ds\Big\}u^0_n\Big|^2\Big)^{1/2}\\
&\leq &
\frac{1}{t}\Big(-\lambda_{n_0}t+\int_0^t\alpha(r(s))ds+\log(\|u^0\|)\Big),
 \nonumber
\end{eqnarray}
since $\lambda_n,n\geq1$, are increasing. On the other hand, we can
also derive  that
\begin{equation}\label{eq8}
\begin{split}
\frac{1}{t}\log(\|v(t)\|)&=
\frac{1}{t}\log\Big(\sum\limits_{n=n_0}^{\infty}
\Big|\exp\Big\{-\lambda_{n_0}t+\int_0^t\alpha(r(s))ds\Big\}u^0_n\Big|^2\Big)^{1/2}\\
&\geq\frac{1}{t}\Big(-\lambda_{n_0}t+\int_0^t\alpha(r(s))ds+\log(|u^0_{n_0}|)\Big).
\end{split}
\end{equation}
 Letting $t\rightarrow\infty$
in both \eqref{eq7} and \eqref{eq8} and taking into account the
ergodic property of  Markov chains, we obtain the  first assertion
\eqref{eq5}.

(2) For any initial data $u^0$,  we observe from \eqref{eq7} that
\begin{equation*}
\frac{1}{t}\log(\|v(t)\|)\leq\frac{1}{t}\Big(-\lambda_1t+\int_0^t\alpha(r(s))ds+\log(\|u^0\|)\Big).
\end{equation*}
Therefore, the ergodic property of Markov chains yields the other
assertion \eqref{eq13}.
\end{proof}

\begin{example}
{\rm Let $r(t), t\geq0$, be a right-continuous Markov chain taking
values in $\mathbb{S}=\{1,2,3\}$ with the generator
\begin{equation*}
\Gamma=\left(
\begin{array}{ccc}
 -2 & 1 & 1 \\
 3 & -4 & 1 \\
 1 & 1 & -2
\end{array}
\right).
\end{equation*}
It is straightforward to see that the unique stationary probability
distribution of the Markov chain $r(t)$ is
\begin{equation*}
\pi=\left(\frac{7}{15},\frac{1}{5},\frac{1}{3}\right).
\end{equation*}
Let us consider  hybrid heat equation
\begin{equation}\label{eq27}
\begin{cases}
\frac{\partial v(t,x)}{\partial t}=A v(t,x)+\alpha(r(t))v(t,x),\
t>0, x\in(0,\pi);\\
v(t,0)=v(t,\pi)=0, t>0; \ \ \ \ v(0,x)= u^0(x), x\in(0,\pi),
\end{cases}
\end{equation}
where $u^0(x)=\sqrt{2/\pi} \sin x$ for $x\in (0, \pi)$.   Let
$\alpha(1)=0.1, \alpha(2)=1.5, \alpha(3)=0.2$. Moreover, we recall
that $e_n(x)=\sqrt{2/\pi}\sin nx,\ n=1,2,3,\cdots,$ are
eigenfunctions of $-A$, with positive and increasing eigenvalues
$\lambda_n=n^2$, and form an orthonormal basis of $L^2(\mathcal
{O})$.

To see what this example shows us, we regard Eq. \eqref{eq27} as the
result of the following three equations
\begin{equation}\label{eq28}
\begin{cases}
\frac{\partial v(t,x)}{\partial t}=A v(t,x)+\frac{1}{10}v(t,x),\
t>0, x\in(0,\pi);\\
v(t,0)=v(t,\pi)=0, \ t>0; \ \  v(0,x)= u^0(x), x\in(0,\pi),
\end{cases}
\end{equation}

\begin{equation}\label{y280}
\begin{cases}
\frac{\partial v(t,x)}{\partial t}=A v(t,x)+\frac{3}{2}v(t,x),\
t>0, x\in(0,\pi);\\
v(t,0)=v(t,\pi)=0, \ t>0; \ \
 v(0,x)= u^0(x), x\in(0,\pi),
\end{cases}
\end{equation}
and
\begin{equation}\label{y0028}
\begin{cases}
\frac{\partial v(t,x)}{\partial t}=A v(t,x)+\frac{1}{5}v(t,x),\
t>0, x\in(0,\pi);\\
v(t,0)=v(t,\pi)=0, \ t>0; \ \
 v(0,x)= u^0(x), x\in(0,\pi),
\end{cases}
\end{equation}
switching from one to another according to the movement of the
Markov chain $r(t)$.  We observe that the solutions of Eq.
\eqref{eq28} and Eq. \eqref{y0028} will converge exponentially to
zero  since the Lyapunov exponents  are $-0.9$ and $-0.8$,
respectively, while the solution of Eq. \eqref{y280} will explode
exponentially  since
 the Lyapunov exponent is $0.5$.  However, as the result of Markovian switching,
 the overall behaviour, i.e. the solution of Eq. \eqref{eq27}
 will converge exponentially to zero,  since,
by Theorem \ref{theorem2},  the  solution of Eq. \eqref{eq27} obeys
\begin{equation*}
\lim \limits_{t\rightarrow\infty}\frac{1}{t}\log(\|v(t)\|) =
-\frac{8}{15}.
\end{equation*}
 }
\end{example}

To get the explicit expression for the  solution of Eq. \eqref{eq1},
we need one more result.

\begin{lemma}
 {\rm \cite[Theorem 5.22, p182]{my06} For any positive integer $n$
and $t\geq0$, hybrid SDE
\begin{equation}\label{eq9}
dz_n(t)=(-\lambda_n+\alpha(r(t)))z_n(t)dt+\beta(r(t))z_n(t)dB(t)
\end{equation}
with initial condition $z_n(0) =u^0_n$ has the explicit solution
\begin{equation}\label{eq12}
\begin{split}
z_n(t)
=u^0_n\exp\Big\{-\lambda_nt+\int_0^t\Big[\alpha(r(s))-\frac{1}{2}\beta^2(r(s))\Big]ds
+\int_0^t\beta(r(s))dB(s)\Big\}.
\end{split}
\end{equation}
}
\end{lemma}

We can now state the explicit formula for the solution of Eq.
\eqref{eq1}.

\begin{theorem}\label{theorem3}
{\rm The unique  solution of Eq. \eqref{eq1} has the explicit form
\begin{equation}\label{eq14}
\begin{split}
u(t,x)=\sum\limits_{n=1}^{\infty}z_n(t)e_n(x)
=v(t,x)\exp\Big\{-\frac{1}{2}\int_0^t\beta^2(r(s))ds
+\int_0^t\beta(r(s))dB(s)\Big\},
\end{split}
\end{equation}
where $v(t,x)$ is the solution to \eqref{eq2} given explicitly by
\eqref{eq6}. }
\end{theorem}

\begin{proof}
Set
\begin{equation}\label{eq11}
\bar{u}(t,x):=\sum\limits_{n=1}^{\infty}z_n(t)e_n(x).
\end{equation}
and note from \eqref{eq00} that
\begin{equation*}
A\bar{u}(t,x)=-\sum\limits_{n=1}^{\infty}\lambda_nz_n(t)e_n(x).
\end{equation*}
Substituting \eqref{eq9} into \eqref{eq11}, we then compute
\begin{equation*}
\begin{split}
\bar{u}(t,x)&=\sum\limits_{n=1}^{\infty}\Big\{u^0_ne_n(x)+\int_0^t(-\lambda_n+\alpha(r(s)))e_n(x)z_n(s)ds\\
&+\int_0^t\beta(r(s))e_n(x)z_n(s)dB(s)\Big\}\\
&=\sum\limits_{n=1}^{\infty}u^0_ne_n(x)-\int_0^t\sum\limits_{n=1}^{\infty}\lambda_ne_n(x)z_n(s)ds+\int_0^t\alpha(r(s))\sum\limits_{n=1}^{\infty}e_n(x)z_n(s)ds\\
&+\int_0^t\beta(r(s))\sum\limits_{n=1}^{\infty}e_n(x)z_n(s)dB(s)\\
&=u^0(x)+\int_0^tA\bar{u}(s,x)ds+\int_0^t\alpha(r(s))\bar{u}(s,x)ds
+\int_0^t\beta(r(s))\bar{u}(,x)dB(s).
\end{split}
\end{equation*}
Consequently, $\bar u$ is a  solution of Eq. \eqref{eq1}. While, by
the uniqueness of  solution of Eq. \eqref{eq1}, we must have
\begin{equation*}
u(t,x)=\sum\limits_{n=1}^{\infty}z_n(t)e_n(x).
\end{equation*}
Taking into account \eqref{eq12}, we obtain the desired explicit
form \eqref{eq14}.
\end{proof}

\section{Sample Lyapunov Exponents}

Making use of the explicit solution of Eq. \eqref{eq1}, we can now
discuss its Lyapunov exponent. Let us now begin with  the sample
Lyapunov exponent.

\begin{theorem}\label{theorem4}
 {\rm The  solution $\{u(t)\}_{t\ge 0}$ of Eq. \eqref{eq1} has the following properties:
 \begin{enumerate}
\item[\textmd{(1)}] If
$u^0$ is deterministic and $u^0\not= 0$, then
\begin{equation*}
\lim\limits_{t\rightarrow\infty}\frac{1}{t}\log(\|u(t)\|)=-\Big(\lambda_{n_0}-\sum\limits_{j=1}^{N}\pi_j\Big(\alpha_j-\frac{1}{2}\beta_j^2\Big)
\Big) \quad \mbox{a.s.}
\end{equation*}
In particular, the solution of Eq. \eqref{eq1} with initial data
$u^0$  will converge exponentially to zero  with probability one if
and only if
\begin{equation*}
\lambda_{n_0}>\sum\limits_{j=1}^{N}\pi_j\Big(\alpha_j-\frac{1}{2}\beta_j^2\Big).
\end{equation*}
\item[\textmd{(2)}]For any initial data $u^0$,
\begin{equation*}
\limsup\limits_{t\rightarrow\infty}\frac{1}{t}\log(\|u(t)\|)\leq-\Big(\lambda_1-\sum\limits_{j=1}^{N}\pi_j\Big(\alpha_j-\frac{1}{2}\beta_j^2\Big)
\Big) \quad\mbox{ a.s.}
\end{equation*}
In particular, the solution  of Eq. \eqref{eq1} is almost surely
exponentially stable if
\begin{equation*}
\lambda_1>\sum\limits_{j=1}^{N}\pi_j\Big(\alpha_j-\frac{1}{2}\beta_j^2\Big).
\end{equation*}
 \end{enumerate}
}
\end{theorem}

\begin{proof}
 It follows from
\eqref{eq14}  that
\begin{equation*}
\begin{split}
\frac{1}{t}\log(\|u(t)\|)&=\frac{1}{t}\Big\{\log(\|v(t)\|)-\frac{1}{2}\int_0^t\beta^2(r(s))ds
+\int_0^t\beta(r(s))dB(s)\Big\}.
\end{split}
\end{equation*}
 By the strong law of large numbers, e.g., \cite[Theorem 1.6,
p16]{my06}
\begin{equation*}
\lim\limits_{t\rightarrow\infty}\frac{1}{t}\int_0^t\beta(r(s))dB(s)=0,
\end{equation*}
while by the ergodic property of Markov chains,
\begin{equation*}
-\frac{1}{2}\lim\limits_{t\rightarrow\infty}\frac{1}{t}\int_0^t\beta^2(r(s))ds=-\frac{1}{2}\sum\limits_{j=1}^N\pi_j\beta_j^2.
\end{equation*}
The desired assertions then follow  from Theorem
\ref{theorem2}.
\end{proof}

In the sequel we set an example to demonstrate our theories and
reveal that an unstable stochastic system can be stabilized by
Markovian switching.

\begin{example}\label{example}
{\rm Let $B(t)$ be a scalar Brownian motion and $r(t)$ a
right-continuous Markov chain taking values in $\mathbb{S}=\{1,2\}$
with the generator $\Gamma=(\gamma_{ij})_{2\times2}$:
\begin{equation*}
-\gamma_{11}=\gamma_{12}>0,\ \ \ \ \ -\gamma_{22}=\gamma_{21}>0.
\end{equation*}
It is  straightforward to show that the unique stationary probability
distribution of the Markov chain $r(t)$ is
\begin{equation*}
\pi=(\pi_1,\pi_2)=\left(\frac{\gamma_{22}}{\gamma_{11}+\gamma_{22}},\frac{\gamma_{11}}{\gamma_{11}+\gamma_{22}}\right).
\end{equation*}
Consider stochastic heat equation with Markovian switching:
\begin{equation}\label{eq15}
\begin{cases}
\frac{\partial u(t,x)}{\partial t}&=A
u(t,x)+\alpha(r(t))u(t,x)+\beta(r(t))u(t,x)\dot{B}(t),\
t>0, x\in(0,\pi);\\
u(t,0)&=u(t,\pi)=0, t>0; \ \ \ \ \ \ \ \ \ u(0,x)=\sqrt{2/\pi}\sin
x,\ \ \ \ x\in(0,\pi),
\end{cases}
\end{equation}
where $\alpha(1)=a, \alpha(2)=b, \beta(1)=c, \beta(2)=d$ with
$a,b,c,d\in\mathbb{R}$. Moreover, we recall that
$e_n(x)=\sqrt{2/\pi}\sin nx, n=1,2,3,\cdots,$ are eigenfunctions of
$-A$, with positive and increasing eigenvalues $\lambda_n=n^2$, and
$e_n\in\mathcal {D}(A)$. Note that initial condition
$u(0,x)=\sqrt{2/\pi}\sin x$ is deterministic and $u_1^0=\langle
\sqrt{2/\pi}\sin x,\sqrt{2/\pi}\sin x\rangle=1$, which implies
$n_0=1$. By Theorem \ref{theorem4}, the unique solution $u(t,x)$
converges exponentially to zero  if and only if
\begin{equation*}
1-(\pi_1\alpha_1+\pi_2\alpha_2)+\frac{1}{2}(\pi_1\beta_1^2+\pi_2\beta_2^2)>0.
\end{equation*}
That is,
\begin{equation}\label{eq17}
\frac{a\gamma_{22}+b\gamma_{11}}{\gamma_{11}+\gamma_{22}}<1+\frac{c^2\gamma_{22}+d^2\gamma_{11}}{2(\gamma_{11}+\gamma_{22})}.
\end{equation}
As a speical case, let us set $a=2, b=1, c=1, d=1$ and
\begin{equation*}
-\gamma_{11}=\gamma_{12}=4,\ \ \ \ \
-\gamma_{22}=\gamma_{21}=\gamma>0.
\end{equation*}
Then the stochastic system \eqref{eq15} can be regarded as the result of
two equations
\begin{equation}\label{eq16}
\frac{\partial u(t,x)}{\partial t}=A
u(t,x)+2u(t,x)+u(t,x)\dot{B}(t),\ t\geq0, x\in(0,\pi)
\end{equation}
and
\begin{equation}\label{eq0}
\frac{\partial u(t,x)}{\partial t}=A
u(t,x)+u(t,x)+u(t,x)\dot{B}(t),\ t\geq0, x\in(0,\pi),
\end{equation}
with the  Dirichlet boundary condition and initial condition,
switching from one to the other according to the law of Markov
chain. By Theorem \ref{theorem4}, the  solution to Eq.\eqref{eq16}
has the property
\begin{equation*}
\lim\limits_{t\rightarrow\infty}\frac{1}{t}\log(\|u(t)\|)=\frac{1}{2},
\end{equation*}
and the solution to Eq.\eqref{eq0}  has the property
\begin{equation*}
\lim\limits_{t\rightarrow\infty}\frac{1}{t}\log(\|u(t)\|)=-\frac{1}{2}.
\end{equation*}
That is, the solution of  stochastic system \eqref{eq16} explodes
exponentially, and the solution of stochastic system \eqref{eq0}
converges exponentially to zero. However, by \eqref{eq17},  the
unique  solution $u(t,x)$ of Eq.\eqref{eq15} converges exponentially
to zero if and only if $0<\gamma<4$.  This shows once again that
Markovian switching plays a key role in the stability of hybrid
stochastic heat equations.}
\end{example}

\section{$p$th Moment Lyapunov Exponents}

Let us now turn to the discussion of the $p$th moment Lyapunov exponent.
The
 $p$th moment exponential stability of stochastic heat equations
without Markov switching has been discussed, e.g., in \cite{clm01,
k99, k02, x08}.   However, due to  Markov switching, it is much more
complicated to study the moment Lyapunov exponent of hybrid
stochastic heat equations.  To cope with the difficulties arising
from Markov switching, large deviation techniques will be used in
this section.

In what follows, we recall some details with respect to large
deviation, see, e.g., Donsker and Varadhan \cite{dv75}. Let
$(X,\mathscr{B},\|\cdot\|_X)$ be a Polish space, and $p(t,x,dy)$ the
transition probability of an $X$-valued Markov process $Z(t)$ with
$Z(0)=x$. Let $P_t$ be a strongly continuous Markovian semigroup
associated with $Z(t)$, and define $P_tf(x):=\int_Xf(y)p(t,x,dy)$
for $f\in C(X)$, the space of continuous functions on $X$. Let $L$
be the infinitesimal generator of the semigroup $P_t$ with domain
$\mathscr{D}(L)$. Let $\mathscr{M}$ be the space of all probability
measures on $X$. For any $\mu\in\mathscr{M}$, define the rate
function by
\begin{equation}\label{eq03}
I(\mu):=-\inf_{u>0,u\in\mathscr{D}}\int_X\left(\frac{Lu}{u}(x)\right)\mu(dx).
\end{equation}
Let $\Omega_x$ be the space of $X$-valued c\`{a}dl\`{a}g functions
$Z(t), 0\leq t<\infty$, with $Z(0)=x$.  For each
$t>0,\omega\in\Omega_x$, and Borel set $A\subset \mathscr{B}$, the
occupation time measure is defined by
\begin{equation}\label{eq01}
L_t(\omega,A):=\frac{1}{t}\int_0^tI_A(Z(s))ds.
\end{equation}
In other words, $L(\omega,A)$ is the proportion of time up to $t$
that a particular sample $\omega=Z(\cdot)$ spends in the set $A$.
Note that for each $t>0$ and each $\omega$, $L(\omega,\cdot)$ is a
probability measure on $X$.

\begin{remark}
{\rm For the Markov chain $r(t)$, we remark that
\begin{equation*}
X=\mathbb{S} \mbox{ and }
L_t(\omega,i)=\frac{1}{t}\int_0^tI_{\{i\}}(r(s))ds, \ \
i\in\mathbb{S}.
\end{equation*}

}
\end{remark}
\begin{remark}
{\rm For a continuous time Markov chain with finite state space and
$Q$-matrices $(q_{ij})_{N\times N}$, it is easy to see that the rate
function has the following expression (see\cite{gq88})
\begin{equation}\label{mrate}
I(\mu)=-\inf\limits_{u_i>0}\sum\limits_{i,j=1}^N\frac{\mu_iq_{ij}u_j}{u_i},
\end{equation}
where $\mu$ is the probability measure on the state space
$\{1,2,\cdots, N\}$.}
\end{remark}

\begin{lemma}\label{large deviation}{\rm  \cite[Theorem 4]{dv75}
If $\Phi$ is a real-valued weakly continuous functional on
$\mathscr{M}$, then
\begin{equation*}
\lim\limits_{t\rightarrow\infty}\frac{1}{t}\log(E\{\exp\{t\Phi(L_t(\omega,\cdot))\}\})=\sup_{\mu\in\mathscr{M}}[\Phi(\mu)-I(\mu)]
\end{equation*}
where $L_t(\omega,\cdot)$ and $I(\mu)$ is defined by \eqref{eq01}
and \eqref{mrate}, respectively.}
\end{lemma}

\begin{theorem}\label{theorem5}
{\rm Let $p > 0$. The solution to Eq. \eqref{eq1} has the following
properties:
\begin{enumerate}
\item[\textmd{(1)}]  If $u^0$ is deterministic and $u^0\not= 0$, then
\begin{equation}\label{eq22}
\lim\limits_{t\rightarrow\infty}\frac{1}{t}\log(
\mathbb{E}(\|u(t)\|^p))=-p\lambda_{n_0}
+\sup\limits_{\mu}\Big\{\sum\limits_{i=1}^N g(i)\mu(i)-I(\mu)\Big\},
\end{equation}
where the supremum is taken over probability measures $\mu$ on
$\mathbb{S}$, $g(i)=p\alpha_i+\frac{p(p-1)}{2}\beta_i^2,
i=1,2,\cdots,N$ and $I(\mu)$ is defined by \eqref{mrate}. In
particular, the $p$th moment of  the solution of Eq. \eqref{eq1}
with initial data $u^0$ will converge exponentially with probability
one to zero if and only if
\begin{equation*}
p\lambda_{n_0} -\sup\Big\{\sum\limits_{i=1}^N
g(i)\mu(i)-I(\mu)\Big\}>0.
\end{equation*}
\item[\textmd{(2)}] For any $u^0$ (which is independent
of $r(\cdot)$ and $B(\cdot)$ as assumed throughout this paper),
\begin{equation}\label{eq23}
\lim\sup\limits_{t\rightarrow\infty}\frac{1}{t}\log(
\mathbb{E}(\|u(t)\|^p))\leq-p\lambda_1
+\sup\Big\{\sum\limits_{i=1}^N g(i)\mu(i)-I(\mu)\Big\}.
\end{equation}
In particular, the solution  of Eq. \eqref{eq1} is exponentially
stable in the $p$th moment if
\begin{equation*}
-p\lambda_1 +\sup\Big\{\sum\limits_{i=1}^N
g(i)\mu(i)-I(\mu)\Big\}<0.
\end{equation*}
\end{enumerate}
}
\end{theorem}

\begin{proof}
Recall that the solution of
Eq. \eqref{eq1} has the explicit form
\begin{equation}
u(t,x)=v(t,x)\exp\Big\{-\frac{1}{2}\int_0^t\beta^2(r(s))ds+\int_0^t\beta(r(s))dB(s)\Big\},
\end{equation}
where
\begin{equation*}
v(t,x)=\sum\limits_{n=1}^{\infty}\exp\Big\{-\lambda_nt+\int_0^t\alpha(r(s))ds
\Big\}u^0_ne_n(x).
\end{equation*}
It is well known that almost every sample path of the Markov chain
$r(\cdot)$ is a right continuous  step function with a finite number
of sample jumps in any finite subinterval of
$\mathbb{R}_+:=[0,\infty)$. Hence there is a sequence of finite
stopping times $0=\tau_0<\tau_1<\cdots<\tau_k\uparrow\infty$ such
that $ r(t)=\sum_{k=0}^{\infty}r(\tau_k)I_{[\tau_k,\tau_{k+1})}(t),
t\geq0. $ For any integer $k>0$ and $t\geq0$, we compute
\begin{equation*}
\begin{split}
\|u(t\wedge\tau_k)\|^p&=\|v(t\wedge\tau_k)\|^p\exp\Big\{-\frac{p}{2}\int_0^{t\wedge\tau_k}\beta^2(r(s))ds+\int_0^{t\wedge\tau_k}p\beta(r(s))dB(s)\Big\}\\
&=\|v(t\wedge\tau_k)\|^p\xi(t\wedge\tau_k)\exp\Big\{-\frac{p^2}{2}\int_0^{t\wedge\tau_k}\beta^2(r(s))ds+\int_0^{t\wedge\tau_k}p\beta(r(s))dB(s)\Big\}\\
&=\|v(t\wedge\tau_k)\|^p\xi(t\wedge\tau_k)\prod\limits_{j=0}^{k-1}\zeta_j(t),
\end{split}
\end{equation*}
where
\begin{equation*}
\xi(t)=\exp\Big\{\frac{p(p-1)}{2}\int_0^{t}\beta^2(r(s))ds\Big\}
\end{equation*}
and
\begin{equation*}
\zeta_j(t)=\exp\Big\{-\frac{1}{2}p^2\beta^2(r(t\wedge\tau_j))(t\wedge\tau_{j+1}-t\wedge\tau_j)+p\beta(r(t\wedge\tau_j))[B(t\wedge\tau_{j+1})
-B(t\wedge\tau_j)]\Big\}.
\end{equation*}
Letting $\mathcal {G}_t=\sigma(\{r(u)\}_{u\geq0},\{B(s)\}_{0\leq
s\leq t})$, by properties of conditional expectations, we have
\begin{equation*}
\begin{split}
\mathbb{E}(\|u(t\wedge\tau_k)\|^p)&=\mathbb{E}\Big(\|v(t\wedge\tau_k)\|^p\xi(t\wedge\tau_k)\prod\limits_{j=0}^{k-1}\zeta_j(t)\Big)\\
&=\mathbb{E}\Big\{\mathbb{E}\Big(\|v(t\wedge\tau_k)\|^p\xi(t\wedge\tau_k)\prod\limits_{j=0}^{k-1}\zeta_j(t)\Big|\mathcal
{G}_{t\wedge\tau_{k-1}}\Big)\Big\}\\
&=\mathbb{E}\Big\{\Big[\|v(t\wedge\tau_k)\|^p\xi(t\wedge\tau_k)\prod\limits_{j=0}^{k-2}\zeta_j(t)\Big]\mathbb{E}\Big(\zeta_{k-1}(t)\Big|\mathcal
{G}_{t\wedge\tau_{k-1}}\Big)\Big\}.
\end{split}
\end{equation*}
Clearly,
\begin{equation*}
\begin{split}
\mathbb{E}(\zeta_{k-1}(t)|\mathcal
{G}_{t\wedge\tau_{k-1}})&=\mathbb{E}\Big(\sum\limits_{i\in\mathbb{S}}I_{\{r(t\wedge\tau_{k-1})=i\}}\zeta_{k-1}^i(t)\Big|\mathcal
{G}_{t\wedge\tau_{k-1}}\Big)\\
&=\sum\limits_{i\in\mathbb{S}}I_{\{r(t\wedge\tau_{k-1})=i\}}\mathbb{E}(\zeta_{k-1}^i(t)|\mathcal
{G}_{t\wedge\tau_{k-1}}),
\end{split}
\end{equation*}
where
\begin{equation*}
\zeta_{k-1}^i(t)=\exp\Big\{-\frac{1}{2}p^2\beta^2_i(t\wedge\tau_{k}-t\wedge\tau_{k-1})+p\beta_i[B(t\wedge\tau_{k})
-B(t\wedge\tau_{k-1})]\Big\},\ \ \ i\in\mathbb{S}.
\end{equation*}
Noting that $t\wedge\tau_k-t\wedge\tau_{k-1}$ is $\mathcal
{G}_{t\wedge\tau_{k-1}}$-measurable and $B(t_2)-B(t_1)$ is
independent of $\mathcal {G}_{t\wedge\tau_{k-1}}$ whenever $t_2\geq
t_1\geq t\wedge\tau_{k-1}$, by \cite[Lemma 3.2, p104]{my06}, we can
derive that
\begin{equation*}
\begin{split}
& \mathbb{E}(\zeta_{k-1}^i(t)|\mathcal
{G}_{t\wedge\tau_{k-1}}) \\
&=\exp\Big\{-\frac{1}{2}p^2\beta^2_i(t\wedge\tau_k-t\wedge\tau_{k-1})\Big\}
\mathbb{E}\exp(p\beta_i[B(t\wedge\tau_k)
-B(t\wedge\tau_{k-1})]|\mathcal
{G}_{t\wedge\tau_{k-1}})\\
&=\exp\Big\{-\frac{1}{2}p^2\beta^2_i(t\wedge\tau_k-t\wedge\tau_{k-1})\Big\}
\mathbb{E}\exp(p\beta_i[B(t_2) -B(t_1)])|_{t_2=t\wedge\tau_k,t_1=t\wedge\tau_{k-1}}\\
&=\exp\Big\{-\frac{1}{2}p^2\beta^2_i(t\wedge\tau_k-t\wedge\tau_{k-1})\Big\}
\exp\Big\{\frac{1}{2}p^2\beta^2_i(t\wedge\tau_k-t\wedge\tau_{k-1})\Big\}\\
&=1.
\end{split}
\end{equation*}
Hence,
\begin{equation*}
\mathbb{E}(\|u(t\wedge\tau_k)\|^p)=\mathbb{E}\Big(\|v(t\wedge\tau_k)\|^p\xi(t\wedge\tau_k)\prod\limits_{j=0}^{k-2}\zeta_j(t)\Big).
\end{equation*}
Repeating this procedure, we obtain that
\begin{equation*}
\mathbb{E}(\|u(t\wedge\tau_k)\|^p)=\mathbb{E}(\|v(t\wedge\tau_k)\|^p\xi(t\wedge\tau_k)).
\end{equation*}
Letting $k\rightarrow\infty$ gives
\begin{equation}\label{eq25}
\mathbb{E}(\|u(t)\|^p)=\mathbb{E}(\|v(t)\|^p\xi(t)).
\end{equation}

\noindent (1) The intial data $u^0$ is deterministic and
$u^0\not=0$. Using \eqref{eq6}, we compute
\begin{equation}\label{eq04}
\begin{split}
\mathbb{E}(\|v(t)\|^p\xi(t))&=\mathbb{E}\Big(\Big\|\sum\limits_{n=1}^{\infty}\exp\Big\{-\lambda_nt+\int_0^t\alpha(r(s))ds\Big\}
u^0_n e_n\Big\|^p\xi(t)\Big)\\
&=\mathbb{E}\Big(\Big\|\sum\limits_{n=n_0}^{\infty}\exp\Big\{-\lambda_nt+\int_0^t\alpha(r(s))ds\Big\}u^0_n e_n\Big\|^p\xi(t)\Big)\\
&=\mathbb{E}\Big(\Big(\sum\limits_{n=n_0}^{\infty}\Big|\exp\Big\{-\lambda_nt+\int_0^t\alpha(r(s))ds\Big\}u^0_n\Big|^2\Big)^{\frac{p}{2}}\xi(t)\Big)\\
&\leq\|u^0\|^p\mathbb{E}\Big(\exp\Big\{-p\lambda_{n_0}t+\int_0^tg(r(s))ds\Big\}\Big),
\end{split}
\end{equation}
and
\begin{equation*}
\mathbb{E}\Big(\|v(t)\|^p\xi(t)\Big)\geq|u_n^0|^p\mathbb{E}\Big(\exp\Big\{-p\lambda_{n_0}t+\int_0^tg(r(s))ds\Big\}\Big).
\end{equation*}
Therefore, by \eqref{eq25}
\begin{equation}\label{eq02}
\lim\limits_{t\rightarrow\infty}\frac{1}{t}\log(
\mathbb{E}(\|u(t)\|^p))=-p\lambda_{n_0}+\lim\limits_{t\rightarrow\infty}\frac{1}{t}\log\Big(\mathbb{E}\Big(\exp\Big\{\int_0^tg(r(s))ds\Big\}\Big)\Big).
\end{equation}
Furthermore, note that
\begin{equation*}
\begin{split}
\int_0^tg(r(s))ds&=\int_0^tg(r(s))1_{\{r(s)\in\mathbb{S}\}}ds\\
&=t\sum\limits_{i=1}^Ng(i)\frac{1}{t}\int_0^t1_{\{i\}}(r(s))ds\\
&=t\sum\limits_{i=1}^Ng(i)L_t(\omega,i).
\end{split}
\end{equation*}
Then  it follows from Lemma \ref{large deviation} that
\begin{equation*}
\lim\limits_{t\rightarrow\infty}\frac{1}{t}\log\Big(\mathbb{E}\Big(\exp\Big\{\int_0^tg(r(s))ds\Big\}\Big)\Big)=\sup\Big\{\sum\limits_{i=1}^N
g(i)\mu(i)-I(\mu)\Big\},
\end{equation*}
where the supremum is taken over probability measures $\mu$ on
$\mathbb{S}$ and $I(\mu)$ is defined by \eqref{eq03}.

\noindent (2) For any $u^0$ which is independent of $r(\cdot)$ and
$B(\cdot)$, we observe from \eqref{eq04} that
\begin{equation*}
\begin{split}
\mathbb{E}(\|v(t)\|^p\xi(t))&\leq\mathbb{E}\Big(\exp\Big\{-p\lambda_1t+\int_0^tg(r(s))ds\Big\}\|
u^0\|^p\Big)\\
&=\mathbb{E}\Big(\exp\Big\{-p\lambda_1t+\int_0^tg(r(s))ds\Big\}\Big)\mathbb{E}(\|
u^0\|^p).
\end{split}
\end{equation*}
Hence the desired assertion \eqref{eq23} follows from Lemma
\ref{large deviation}.
\end{proof}

\begin{remark}
{\rm Since the Markov chain $r(t)$ with $Q$-matrix
$(\gamma_{ij})_{N\times N}, i,j\in\mathbb{S}$ is assumed to be
irreducible throughout the paper, there exists a unique invariant
measure $\pi=(\pi_1,\pi_2,\cdots,\pi_N)$ such that
\begin{equation*}
\sum_{j=1}^N\pi_j\gamma_{ij}=0 \mbox{ and } \sum_{j=1}^N\pi_j =1.
\end{equation*}
Letting $\mu$ be the invariant measure $\pi$ and
$u_i,i\in\mathbb{S}$ be constants, together with the nonnegative
property of rate function, it follows that $I(\pi)$=0. Therefore,
\begin{equation*}
\lim\limits_{t\rightarrow\infty}\frac{1}{t}\log(
\mathbb{E}(\|u(t)\|^p))\geq-p\lambda_{n_0} +\sum\limits_{i=1}^N
g(i)\pi(i)
\end{equation*}
whenever $u_0$ is deterministic.}
\end{remark}

\begin{remark}
 {\rm Let us reconsider  Example \ref{example}, where $r(t)$ is a
right-continuous Markov chain taking values in state space
$\mathbb{S}=\{1,2\}$ with the generator
$\Gamma=(\gamma_{ij})_{2\times2}$:
\begin{equation*}
-\gamma_{11}=\gamma_{12}=1,\ \ \ \ \ -\gamma_{22}=\gamma_{21}=q>0.
\end{equation*}
By computation, the invariant measure of the Markov chain is
$\pi=(\pi_1,\pi_2)=(q/(q+1),1/(q+1))$. In what follows, we assume
that, for some probability measure $\mu_0=(\theta,1-\theta),
0<\theta<1$, on the state space $\mathbb{S}$,
\begin{equation*}
a\theta+b(1-\theta)-I(\mu_0)>a\pi_1+b\pi_2,
\end{equation*}
where $g(1)=a>0, g(2)=b>0$ and $g(i),i=1,2$ is defined in Theorem
\ref{theorem5}, that is,
\begin{equation*}
\begin{split}
(a-b)\theta+b&>a\pi_1+b\pi_2+\theta+(1-\theta)q-\inf\limits_{u_1,u_2>0}\Big[\theta\frac{u_2}{u_1}+(1-\theta)q\frac{u_1}{u_2}\Big]\\
&=a\pi_1+b\pi_2+\theta+(1-\theta)q-2\sqrt{\theta(1-\theta)q}.
\end{split}
\end{equation*}
Recalling $\pi_1=q/(q+1), \pi_2=1/(q+1)$ and letting
$\mu=\frac{1}{2}$, then the previous inequality can be rewritten as
\begin{equation}\label{00}
a+b>(q^2-2q^{\frac{3}{2}}+2q-2q^{\frac{1}{2}}+1)/(3q+1).
\end{equation}
Since the right hand side of the inequality above tends to zero
whenever $q\rightarrow1$, then we can choose suitable $q$ to satisfy
\eqref{00} and further get
\begin{equation*}
-p\lambda_{n_0}+\sup\limits_{\mu}\Big\{\sum\limits_{i=1}^2
g(i)\mu(i)-I(\mu)\Big\}>-p\lambda_{n_0}+\sum\limits_{i=1}^2
g(i)\pi_i-I(\pi).
\end{equation*}
Since $I(\pi)=0$ and $\lambda_{n_0}=1$, the  solution of Eq.
\eqref{eq15} will explode exponentially with probability one to
$\infty$ provided that $ -p+\sum_{i=1}^2 g(i)\pi_i>0. $ }
\end{remark}

So far we have assumed the Brownian motion is a scalar one.  However,
our theory  can
 easily be generalized to cope with a hybrid stochastic heat equation
driven by a multi-dimensional Brownian motion of the form as Xie \cite{x08}.
\begin{equation}\label{eq33}
\begin{cases}
\frac{\partial u(t,x)}{\partial t}=A
u(t,x)+\alpha(r(t))u(t,x)+\sum\limits_{i=1}^m\beta_i(r(t))u(t,x)\dot{B}_i(t),
\ t\geq0, x\in\mathcal {O};\\
u(t,x)=0, \ t\geq0, x\in\partial\mathcal {O}; \ \ \ \ \ \ \ \ \
u(0)=u^0, 
\end{cases}
\end{equation}
where $(B_1(t)),\cdots,B_m(t))$ is an $m$-dimensional Brownian
motion, $\alpha,\beta_i$ are mappings from
$\mathbb{S}\rightarrow\mathbb{R}$ and we write
$\beta_{ij}=\beta_i(j)$ for any $j\in\mathbb{S}$.  Following the
similar arguments to that of Theorems \ref{theorem4} and
\ref{theorem5}, we can derive the following theorems.

\begin{theorem}
{\rm The solution $\{u(t)\}_{t\ge 0}$ of Eq. \eqref{eq33} has the
following properties:\begin{enumerate}
\item[\textmd{(1)}] If $u^0$ is deterministic and $u^0\not= 0$, the sample Lyapunov exponent
\begin{equation*}
\lim\limits_{t\rightarrow\infty}\frac{1}{t}\log(\|u(t)\|)=-\Big(\lambda_{n_0}-\sum\limits_{i=1}^{N}\pi_i\Big(\alpha_i
-\frac{1}{2}\sum\limits_{j=1}^m\beta_{ij}^2\Big)\Big) \quad
\mbox{a.s.}
\end{equation*}
 In
particular, the  solution of Eq. \eqref{eq33} with initial data
$u^0$ will converge to zero exponentially with probability one
 if and only if
\begin{equation*}
\lambda_{n_0}>\sum\limits_{i=1}^{N}\pi_i\Big(\alpha_i
-\frac{1}{2}\sum\limits_{j=1}^m\beta_{ij}^2\Big).
\end{equation*}
\item[\textmd{(2)}] For any $u^0$, the sample Lyapunov exponent
\begin{equation*}
\lim\sup\limits_{t\rightarrow\infty}\frac{1}{t}\log(\|u(t)\|)\leq-\Big(\lambda_1-\sum\limits_{i=1}^{N}\pi_i\Big(\alpha_i
-\frac{1}{2}\sum\limits_{j=1}^m\beta_{ij}^2\Big)\Big) \quad\mbox{
a.s.}
\end{equation*}
 In
particular, the solution of Eq.\eqref{eq33} is almost surely
exponentially stable if
\begin{equation*}
\lambda_1>\sum\limits_{i=1}^{N}\pi_i\Big(\alpha_i
-\frac{1}{2}\sum\limits_{j=1}^m\beta_{ij}^2\Big).
\end{equation*}
\end{enumerate}
}
\end{theorem}

\begin{theorem}
{\rm The  solution $\{u(t)\}_{t\ge 0}$  of Eq. \eqref{eq33} has the
following properties:\begin{enumerate}
\item[\textmd{(1)}]  If $u^0$ is deterministic and $u^0\not= 0$,  the
$p$th moment Lyapunov exponent is
\begin{equation*}
\lim\limits_{t\rightarrow\infty}\frac{1}{t}\log(\mathbb{E}(\|u(t)\|^p))=-p\lambda_{n_0}+\sup\limits_{\mu}\Big\{\sum\limits_{i=1}^N\Big(p\alpha
+\frac{p(p-1)}{2}\sum\limits_{j=1}^m\beta_{ij}^2\Big)\mu(i)-I(\mu)\Big\},
\end{equation*}
where the supremum is taken over probability measures $\mu$ on
$\mathbb{S}$ and $I(\mu)$ is defined by \eqref{eq03}. In particular,
the $p$th moment of the  solution of Eq. \eqref{eq33}
 with initial data $u^0$ will  converge exponentially to zero
 if and only if
\begin{equation*}
p\lambda_{n_0}-\sup\limits_{\mu}\Big\{\sum\limits_{i=1}^N\Big(p\alpha
+\frac{p(p-1)}{2}\sum\limits_{j=1}^m\beta_{ij}^2\Big)\mu(i)-I(\mu)\Big\}>0.
\end{equation*}
\item[\textmd{(1)}]  For any $u^0$, the $p$th moment Lyapunov exponent
\begin{equation*}
\lim\sup\limits_{t\rightarrow\infty}\frac{1}{t}\log(\mathbb{E}(\|u(t)\|^p))\leq-p\lambda_1+\sup\limits_{\mu}\Big\{\sum\limits_{i=1}^N\Big(p\alpha
+\frac{p(p-1)}{2}\sum\limits_{j=1}^m\beta_{ij}^2\Big)\mu(i)-I(\mu)\Big\}.
\end{equation*}
In particular, the  solution of Eq. \eqref{eq33} is exponentially
stable in $p$th moment if
\begin{equation*}
-p\lambda_1+\sup\limits_{\mu}\Big\{\sum\limits_{i=1}^N\Big(p\alpha
+\frac{p(p-1)}{2}\sum\limits_{j=1}^m\beta_{ij}^2\Big)\mu(i)-I(\mu)\Big\}<0.
\end{equation*}
\end{enumerate}
}
\end{theorem}

\end{document}